\documentclass[a4paper]{amsart}
\usepackage{amsthm, dsfont, a4}
\usepackage{hyperref, xcolor}
\usepackage[all]{xy}
\usepackage[active]{srcltx}
\usepackage[short,nodayofweek]{datetime}
\usepackage[utf8]{inputenc}

\definecolor{my-blue}{cmyk}{1,0.6,0,0}

\definecolor{my-green}{cmyk}{0.8,0,1,0.5}

\hypersetup{
colorlinks=true, 
linkcolor=my-blue, 
citecolor=my-green,
urlcolor=black
}

\newcommand\EE{{\mathbb E}}
\newcommand\FF{{\mathbb F}}

\newcommand\NN{{\mathbb N}}

\newcommand\TT{{\mathbb T}}

\newcommand{\KI}{K_\infty}
\newcommand{\Ksep}{K^{\rm sep}}
\newcommand{\KIsep}{\KI^{\rm sep}}  

\newcommand{\Fq}{\FF_{\!q}}
\newcommand{\cL}{\mathcal{L}}

\newcommand\pitilde{\tilde{\pi}}
\newcommand\id{{\rm id}}

\newcommand{\transp}[1]{#1^{\rm tr}}   

\newcommand{\hd}[2]{\partial^{(#1)}\!\!\left(#2\right)}
\newcommand{\hde}[1]{\partial^{(#1)}}  
\newcommand{\hdt}[2]{\partial_t^{(#1)}\!\!\left(#2\right)}
\newcommand{\hdte}[1]{\partial_t^{(#1)}}  
\newcommand{\hdth}[2]{\partial_{\theta}^{(#1)}\!\!\left(#2\right)}
\newcommand{\hdthe}[1]{\partial_{\theta}^{(#1)}}  
\newcommand{\D}{\mathcal{D}} 
\newcommand{\Dt}{\mathcal{D}_t} 
\newcommand{\Dth}{\mathcal{D}_\theta} 

\newcommand{\ps}[1]{[\![#1]\!]}  
\newcommand{\ls}[1]{(\!(#1)\!)}
\newcommand{\cs}[1]{\{\!\{#1\}\!\}}

\def\betr#1{\lvert #1\rvert_\infty}

\DeclareMathOperator{\ch}{char}
\DeclareMathOperator{\Mat}{Mat}

\theoremstyle{plain}
\newtheorem{thm}{Theorem}[section]
\newtheorem{cor}[thm]{Corollary}
\newtheorem{lem}[thm]{Lemma}
\newtheorem{prop}[thm]{Proposition}

\newtheorem*{thm*}{Theorem}

\theoremstyle{definition}
\newtheorem{defn}[thm]{Definition}
\newtheorem{exmp}[thm]{Example}
\newtheorem{rem}[thm]{Remark}

\newenvironment{ack}{\textbf{Acknowledgement}}{}

\begin{document}

\title[Hyperderivatives of Carlitz period]{Algebraic independence of the Carlitz period and its hyperderivatives.}
\author{Andreas Maurischat}
\address{\rm {\bf Andreas Maurischat}, Lehrstuhl A f\"ur Mathematik, RWTH Aachen University \& FH Aachen University of Applied Sciences, Germany }
\email{\sf maurischat@fh-aachen.de}
\date{$2^{\rm nd}$ Dec, 2021}



\begin{abstract}
This paper deals with the fundamental period $\tilde{\pi}$ of the Carlitz module. The main theorem states that the Carlitz period 
and all its hyperderivatives are algebraically independent over the base field $\Fq(\theta)$. Our approach also reveals a connection of these hyperderivatives with the coordinates of a period lattice generator of the tensor powers of the Carlitz module which was already observed by M.~Papanikolas in a yet unpublished paper. Namely, these coordinates can be obtained by explicit polynomial expressions in $\tilde{\pi}$ and its hyperderivatives. Papanikolas also gave various presentations of these expressions  which we also prove here.
\end{abstract}

\maketitle

\setcounter{tocdepth}{1}
\tableofcontents

\section*{Introduction}

Periods of Drinfeld modules and Anderson $t$-modules play a central role in number theory in positive characteristic, and
questions about their algebraic independence are of major interest.
The most prominent period is the Carlitz period
\[ \tilde{\pi}=\lambda_\theta \theta \prod_{j \geq 1} (1 - \theta^{1-q^j})^{-1} \in \KI(\lambda_\theta), \]
where $\lambda_\theta\in \Ksep$ is a $(q-1)$-th root of $-\theta$.
Here, $K=\Fq(\theta)$ is the rational function field over the finite field $\Fq$, $\Ksep$ its separable algebraic closure, and $\KI=\Fq\ls{\frac{1}{\theta}}$ is the completion of $K$ with respect to the absolute value
$|\cdot|_\infty$ given by $|\theta|_\infty=q$. 

The Carlitz period is the function field analog of the complex number $2\pi i$, and it was already proven 
by Wade in 1941 that $\pitilde$ is transcendental over $K$ (see \cite{liw:cqtg}).

On the field $\KI(\lambda_\theta)$, one can consider the hyperdifferential operators with respect to $\theta$, denoted by $\hdthe{n}$, $n\geq 0$, which are defined for $\sum_{i=i_0}^\infty c_i \theta^{-i}\in \KI=\Fq\ls{\frac{1}{\theta}}$ by
\[  \hdth{n}{\sum_{i=i_0}^\infty c_i \theta^{-i}} = \sum_{i=i_0}^\infty c_i \binom{-i}{n}\theta^{-i-n}, \]
and are uniquely extended to $\KI(\lambda_\theta)$.

Brownawell and Denis considered hyperderivatives of periods and quasi-periods of Drinfeld modules, and showed linear independence and algebraic independence results for them (see \cite{ld:dmdt}, \cite{ld:iac2}, \cite{db:lidddm-I}, \cite{ld:iadpmc}, \cite{db-ld:lidddm-II}). With respect to the hyperderivatives of the Carlitz period $\pitilde$, they proved the following.

\begin{thm*}
\begin{enumerate}
\item (see \cite[Thm.~1.1]{db-ld:lidddm-II}, case $d=1$)\\
The elements $1,\pitilde, \hdth{1}{\pitilde},\hdth{2}{\pitilde},\ldots $ are $\bar{K}$-linearly independent, where $\bar{K}$ is an algebraic closure of $K$.
\item (see \cite[Thm.~1]{ld:iadpmc})\\
The elements $\pitilde, \hdth{1}{\pitilde},\ldots, \hdth{p-1}{\pitilde}$ are algebraically independent over $K$, where $p=\ch(\Fq)$ is the characteristic of $\Fq$.
\end{enumerate}
\end{thm*}

In this paper, we prove a much stronger result.

\begin{thm} (see Theorem~\ref{thm:pitilde-hypertrancendental})\\
The Carlitz period $\pitilde$ is hypertranscendental over $K=\Fq(\theta)$, i.e.~the set $\{\hdth{n}{\pitilde} \mid n\geq 0\}$ is algebraically independent over $K$.
\end{thm}

Recently, Namoijam proved an even more general result on hypertranscendence of periods and quasi-periods of Drinfeld modules using differential algebraic geometry in positive characteristic (see \cite[Thm.~1.1.5]{cn:arhpldm}). In the case of the Carlitz module, her Theorem 1.1.5 specializes to Theorem~\ref{thm:pitilde-hypertrancendental}. In the second main theorem of \cite{cn:arhpldm}, Namoijam extended Theorem 1.1.5 further to also include logarithms of algebraic points.

Our proof of Theorem~\ref{thm:pitilde-hypertrancendental} is similar to the proof of Theorem~8.1 in \cite{am:ptmaip} where we proved algebraic independence for the coordinates of a fundamental period $\pitilde_n=\transp{(z_1,\ldots, z_n)}$ of the $n$-th Carlitz tensor power $C^{\otimes n}$ ($n\geq 1$) if $n$ is prime to the characteristic.
Furthermore, these two proofs reveal a link between the hyperderivatives of $\pitilde$ and the coordinates of $\pitilde_n$.

\begin{thm} (see Thm.~\ref{thm:coordinates-in-k-space})\\
The coordinates $z_1,\ldots, z_n$ belong to the $K$-vector space generated by the set of ``monomials''
$\left\{ \prod_{j=1}^n \hdth{m_j}{\pitilde} \,\middle|\, \forall j: 0\leq m_j\leq n-1 \right\}. $
\end{thm}
The proof even provides an explicit description for these coordinates.

Such a link between the coordinates of $\pitilde_n$ and the hyperderivatives of $\pitilde$ has already been discovered by Papanikolas using \cite[Lemma 8.3]{am:ptmaip} and results from a yet unpublished manuscript \cite{mp:latpcmsvgl}.
Papanikolas obtained the following explicit description. For stating it, we consider the $t$-linear extensions of the hyperdifferential operators $\hdthe{n}$ to $\KI(\lambda_\theta)(t)$ which we still denote by $\hdthe{n}$.

\begin{thm}\label{thm:papanikolas-identity} (Papanikolas)\\
Let $\pitilde_n$ be chosen such that its last coordinate $z_n$ equals $\pitilde^n$.
For $l\geq 1$, let $\alpha_l(t)\in \Fq[\theta,t]$ be the $l$-th Anderson-Thakur polynomial, and $\Gamma_l$ the $l$-th Carlitz 
factorial. Then for all $j=0,\ldots, n-1$, one has
\[ z_{n-j}= \hdth{j}{\frac{\alpha_{n}(t)}{\Gamma_{n}}\cdot \pitilde^n}\Big|_{t=\theta}. \]
\end{thm}
(See Definition \ref{def:gamma-et-al} for precise definitions).

The proof of this identity, however, is quite long, and we give a shorter proof for it in Section \ref{sec:nicer-expressions}. The main improvement with respect to Papanikolas' proof is the following new identity in the ring $K\cs{t-\theta}$ of convergent power series in $(t-\theta)$.

\begin{thm}\label{thm:eta-in-intro} (see Thm.~\ref{thm:properties-of-eta} \eqref{item:identity-for-eta})\\
For $m\geq 0$, let $ \gamma_m(t)=\prod_{k=1}^m (\theta^{q^m}-t^{q^k})\in \Fq[\theta,t]$,
and
$ D_m = \prod_{k=0}^{m-1} \left(\theta^{q^m}-\theta^{q^{k}}\right)\in \Fq[\theta]$.
Further, let 
\[\eta = \prod_{m=1}^\infty \left( 1+\frac{(t-\theta)^{q^m}}{\theta^{q^m}-\theta}\right)\in K\cs{t-\theta}.\] 
Then
\begin{equation*} 
 \sum_{j=0}^\infty \frac{\gamma_j(t)}{D_j}\cdot \eta^{q^j} = 1\in K\cs{t-\theta}.
\end{equation*}
\end{thm}

Other improvements were obtained by a consequent use of Taylor series expansions and identities for such expansions.

\medskip

In Section \ref{sec:basics}, we set up the notation  and explain the basics on hyperdifferential operators and Taylor series expansions which are used later.
Section \ref{sec:hypertranscendence} is devoted to our main theorem and its proof. We continue by providing the connection to the coordinates of the period of Carlitz tensor powers in Section \ref{sec:tensor-power}. In Section \ref{sec:nicer-expressions}, we present our proofs of the explicit expressions for these period coordinates.

\medskip

\begin{ack}
We thank Matt Papanikolas for sharing his manuscript and for explaining his proof of Theorem \ref{thm:papanikolas-identity}. We also thank the referee for helpful comments and for pointing out some inaccuracies.
\end{ack}

\section{Notation and calculation rules}\label{sec:basics}

\subsection{Basic rings and fields}

Let $\Fq$ be the finite field with $q$ elements and characteristic $p$, and $K=\Fq(\theta)$ the rational function field in the variable $\theta$. 
We equip $K$ with the absolute value $|\cdot |_\infty$ which is given by $|\theta|_\infty=q$.
We further denote by $t$ a second indeterminate. All rings and fields occurring will be extensions of $K$ or of $K[t]$. These are

\noindent \begin{tabular}{p{2.2cm}p{10cm}}
$\KI=\Fq\ls{\frac{1}{\theta}}$& the completion of $K$ at this infinite place,\\
$\KIsep$& a separable algebraic closure of $\KI$ with the extension of the absolute value $|\cdot |_\infty$,\\
$\Ksep$& the separable algebraic closure of $K$ inside $\KIsep$,\\
$K(t)$& the rational function field over $K$,\\
$K[t]_{(t-\theta)}$ & the localization of $K[t]$ at the prime ideal $(t-\theta)$, i.e.~the rational functions in $K(t)$ which are regular at $t=\theta$,\\
$\KIsep\ps{t}$& the power series ring in $t$ with coefficients in $\KIsep$,\\
$\TT_\theta$ & the subring of $\KIsep\ps{t}$ consisting of series which converge on the closed disc of radius $\betr{\theta}=q$.\\
$\EE$& the subring of $\KIsep\ps{t}$ consisting of entire functions, i.e.~of those series $f(t)$ for which $f(x)$ converges for any $x\in \KIsep$, and whose coefficients lie in a finite extension of $\KI$,\\
$K\cs{t-\theta}$& the ring of power series $\sum_{n=0}^\infty a_n(t-\theta)^n$ with coefficients in $K$, and whose coefficients $a_n$ tend to zero as $n$ goes to infinity. \\
\end{tabular}

For elements $f$ in all these rings, we sometimes write $f(t)$ to emphasize the dependence on $t$, or even $f(\theta,t)$, to emphasize both dependencies. If we replace the variable $t$ of such a function $f$ by $z$ for some element $z$, we will denote this by $f|_{t=z}$ or by $f(\theta,z)$. Similarly, if we replace the variable $\theta$ by $z$, we write $f|_{\theta=z}$ or $f(z,t)$.

\subsection{Hyperdifferential operators}

In this subsection, we give a short presentation of hyperdifferential operators, also called higher derivations (see e.g.~\cite[\S 27]{hm:crt}, or \cite{kc:dp}).

\begin{defn}
A higher derivation on an $\Fq$-algebra $R$ is a family of $\Fq$-linear operators $(\hde{n}:R\to R)_{n\geq 0}$ such that
\begin{enumerate}
\item $\hde{0}=\id_R$,
\item $\forall n\geq 0$, $\forall f,g\in R$: $\hd{n}{f\cdot g} = \sum\limits_{i+j=n} \hd{i}{f}\cdot \hd{j}{g}$ \textit{(generalized Leibniz rule)}
\end{enumerate}
\end{defn}

\begin{rem}
By definition, the hyperdifferential operator $\partial:=\hde{1}$ is a derivation, and 
$\hd{n}{f}$ corresponds to $\tfrac{1}{n!}\partial^n(f)$ in characteristic zero.
\end{rem}

For many calculations, it is much more convenient to consider the corresponding map
\[ \D: R\to R\ps{X}, f\mapsto \sum_{n \geq 0} \hd{n}{f} X^n, \]
also called the \textit{(generic) Taylor series expansion}. Namely, the $\Fq$-linearity and the general Leibniz rule imply that the map $\D$ is a homomorphism of $\Fq$-algebras (actually is equivalent to that condition). In particular, $\D$ is determined by the images of generators of the $\Fq$-algebra $R$.

\begin{prop}\label{prop:unique-extensions} (see \cite[Thm.~5 \& 6]{kc:dp})\\
Higher derivations can be uniquely extended to localizations and to separable algebraic field extensions.\\
If $R$ is equipped with an absolute value, and all $\hde{n}$ are continuous, the $\hde{n}$ and also $\D$ can be continuously extended to the completion of $R$ with respect to the absolute value in a unique way.
\end{prop}

The higher derivations/hyperdifferential operators that are relevant in this paper are the following.

\begin{exmp}
On the polynomial ring $\Fq[\theta,t]$, we consider two families of hyperdifferential operators, the  hyperdifferential operators $(\hdthe{n})_{n\geq 0}$  with respect to $\theta$ -- with corresponding generic Taylor series expansion $\Dth$ --, and the  hyperdifferential operators $(\hdte{n})_{n\geq 0}$  with respect to $t$ -- with corresponding generic Taylor series expansion $\Dt$.
They are determined by
\[  \Dth(\theta)=\theta+X,\,\, \Dth(t)=t \quad\text{and} \quad \Dt(\theta)=\theta,\,\, \Dt(t)=t+X. \]
Explicitly, $(\hdthe{n})_{n\geq 0}$ and $(\hdte{n})_{n\geq 0}$ are given by
\begin{eqnarray*}
\hdth{n}{\sum_{i,j} c_{ij}\theta^it^j} &=& \sum_{i,j} \binom{i}{n} c_{ij}\theta^{i-n}t^j, \\
\hdt{n}{\sum_{i,j} c_{ij}\theta^it^j} &=& \sum_{i,j} \binom{j}{n} c_{ij}\theta^{i}t^{j-n}.
\end{eqnarray*}
The maps $\Dth$ and $\Dt$ are extended (by Prop.~\ref{prop:unique-extensions} in a unique way) to the localization $K(t)$, as well as continuously to $\KIsep\ps{t}$ and $K\cs{t-\theta}$ as in \cite{kc:dp}.

This means, $\Dth$ is given on these rings by replacing the variable $\theta$ by $\theta+X$, and expanding the expression into a power series in $X$. Similar, $\Dt$ is given by replacing the variable $t$ by $t+X$, and expanding the expression into a power series in $X$.

It is not hard to see that all the rings mentioned above are stable under both families of hyperdifferential operators.
\end{exmp}

\begin{rem}
Given a generic Taylor series expansion $\D:R\to R\ps{X}$ on an $\Fq$-algebra $R$, we also denote by $\D$ its $X$-linear extension
\[ \D:R\ps{X}\to R\ps{X}. \]
This will be used, when we consider compositions of hyperdifferential operators as in the following lemma.
\end{rem}

\begin{lem}\label{lem:composition-rules}
Let $R$ be one of the rings above.
For all $f(\theta,t)\in R$, one has
\begin{enumerate}
\item $\forall n,m\geq 0$: $\hdt{n}{\hdth{m}{f}}=\hdth{m}{\hdt{n}{f}}$, or equivalently\\ $\Dth(\Dt(f))=\Dt(\Dth(f))$.
\item If we can evaluate $t$ at $\theta$, i.e.~$R$ is a subring of $K[t]_{(t-\theta)}$, $\TT_\theta$ or $K\cs{t-\theta}$, then for all $n\geq 0$:
\[  \hdth{n}{f(\theta,\theta)} = \sum_{i+j=n} \hdt{i}{\hdth{j}{f(\theta,t)}}|_{t=\theta}. \]
Equivalently, in terms of $\Dth$ and $\Dt$:\\
\[  \Dth\left(f(\theta,t)|_{t=\theta}\right) = \Dt\left(\Dth(f(\theta,t)\right)|_{t=\theta}. \]
\end{enumerate}
\end{lem}

\begin{proof}
We check the equations using $\Dt$ and $\Dth$. For the first part one has:
\[ \Dth(\Dt(f(\theta,t)))=\Dth(f(\theta,t+X))=f(\theta+X,t+X)= \Dt(\Dth(f(\theta,t))).\]
For the second part:
\[ \Dth\left(f(\theta,t)|_{t=\theta}\right) = f(\theta+X,\theta+X)= f(\theta+X,t+X)|_{t=\theta}=\Dt\left(\Dth(f(\theta,t)\right)|_{t=\theta}. \]
\end{proof}

We will often need the hyperdifferential operators just up to some bound $n$, i.e.~$\hde{0},\hde{1},\ldots, \hde{n}$. For this it is convenient to transform the corresponding Taylor series homomorphism $\D$ to a homomorphism into a matrix ring.

\begin{defn} (cf.~\cite{am:ptmaip} or \cite{am-rp:iddbcppte})\\
Let $\D:R\to R\ps{X}$ be the Taylor series homomorphism corresponding to a higher derivation $(\hde{k})_{k\geq 0}$ on $R$, let $n\geq 0$, and let $N\in \Mat_{(n+1)\times (n+1)}(R)$ be the nilpotent matrix
\[ N=  \begin{pmatrix}
 0 & 1 & 0 & \cdots & 0 \\ \vdots & \ddots & \ddots & \ddots & \vdots  \\ \vdots && \ddots & \ddots & 0 \\ \vdots && & \ddots & 1 \\
 0 &  \cdots &\cdots &\cdots & 0
\end{pmatrix}. \]
We define the ring homomorphism $\rho_{[n]}:R\to \Mat_{(n+1)\times (n+1)}(R)$ to be the composition of $\D$ with the evaluation homomorphism replacing $X$ by $N$, i.e.~
\begin{equation}\label{eq:rho_n}
 \rho_{[n]}(f) :=\D(f)|_{X=N}= \begin{pmatrix} f & \hd{1}{f} & \cdots & \hd{n}{f} \\ 0 & f & \ddots & \vdots \\ \vdots & \ddots & \ddots & \hd{1}{f} \\ 0 & \cdots & 0 & f \end{pmatrix}.
\end{equation}

In case of the hyperdifferential operators with respect to $\theta$ and $t$, we add the variable as subscript, i.e.~ we denote
\[ \rho_{\theta,[n]}(f) :=\Dth(f)|_{X=N}\quad \text{as well as}\quad \rho_{t,[n]}(f) :=\Dt(f)|_{X=N}
.\]
\end{defn}

\section{Hypertranscendence of the Carlitz period}\label{sec:hypertranscendence}

The main result of this article is on the Carlitz period
\[ \tilde{\pi}=\lambda_\theta \theta \prod_{j \geq 1} (1 - \theta^{1-q^j})^{-1} \in \KI(\lambda_\theta), \]
where $\lambda_\theta\in \Ksep$ is a $(q-1)$-th root of $-\theta$, as well as on its hyperderivatives with respect to $\theta$.

\begin{thm}\label{thm:pitilde-hypertrancendental}
The Carlitz period $\pitilde$ is hypertranscendental over $K=\Fq(\theta)$, i.e.~the set $\{\hdth{n}{\pitilde} \mid n\geq 0\}$ is algebraically independent over $K$.
\end{thm}

The proof of this theorem will take up the whole section.
For the proof, we need the rigid analytic trivialization of the dual Carlitz motive. This is the
$\Omega$-function,
\[  \Omega(t)= \lambda_\theta^{-q} \prod_{j \geq 1} (1 - \frac{t}{\theta^{q^j}}) \in \KI(\lambda_\theta)\ps{t}, \]
which is an entire function, and satisfies
\[ \Omega(\theta)=-\frac{1}{\pitilde}.\]

\begin{lem}\label{lem:theta-hyperderivatives-of-Omega}
For all $j\geq 0$, there is $b_j\in K[t]_{(t-\theta)}$ such that $\hdth{j}{\Omega^{-1}}=b_j\cdot \Omega^{-1}$. 
These $b_j$ are given as $b_0=1$, and for $j>0$,
\begin{equation}\label{eq:b_j-simpler}
 b_j = \prod_{i=1}^{l-1} (\theta^{q^i}-t) \cdot \hdth{j}{\prod_{i=1}^{l-1} (\theta^{q^i}-t)^{-1} },
\end{equation}
where $l=\lfloor\log_q(j)\rfloor+1$.\footnote{As usual, empty products are always considered to be equal to $1$.}
\end{lem}

\begin{proof}
The case $j=0$ is trivial. So let $j>0$ and let $l\in \NN$ such that $q^l>j$ (e.g.~$l=\lfloor\log_q(j)\rfloor+1$). As $\hdth{i}{f^{q^l}}=0$ for all $f\in \KIsep$, $0<i<q^l$,\footnote{This follows immediately from $\Dth(f^{q^l})=\Dth(f)^{q^l}\in
\KIsep\ps{X^{q^l}}$.} one obtains from the generalized Leibniz rule
\begin{eqnarray*}
\hdth{j}{\Omega^{-1}} &=& \hdth{j}{\lambda_\theta^{q-q^l} \prod_{i=1}^{l-1} \left(1-\frac{t}{\theta^{q^i}}\right)^{-1}\cdot \lambda_\theta^{q^l} \prod_{i=l}^\infty \left(1-\frac{t}{\theta^{q^i}}\right)^{-1} }\\
&=& \hdth{j}{\lambda_\theta^{q-q^l} \prod_{i=1}^{l-1} \left(1-\frac{t}{\theta^{q^i}}\right)^{-1} }\cdot \lambda_\theta^{q^l}\prod_{i=l}^\infty \left(1-\frac{t}{\theta^{q^i}}\right)^{-1}\\
&=&   \hdth{j}{\lambda_\theta^{q-q^l} \prod_{i=1}^{l-1} \left(1-\frac{t}{\theta^{q^i}}\right)^{-1} }\cdot \lambda_\theta^{q^l-q} \prod_{i=1}^{l-1} \left(1-\frac{t}{\theta^{q^i}}\right)\cdot \Omega^{-1}.
\end{eqnarray*}
As $q^l-q$ is divisible by $q-1$, we have $\lambda_\theta^{q^l-q}=(-\theta)^{\frac{q^l-q}{q-1}}\in K$, and
hence, $b_j=(-\theta)^{\frac{q^l-q}{q-1}} \prod_{i=1}^{l-1} \left(1-\frac{t}{\theta^{q^i}}\right)\cdot  \hdth{j}{(-\theta)^{\frac{q^l-q}{q-1}} \prod_{i=1}^{l-1} \left(1-\frac{t}{\theta^{q^i}}\right)^{-1} }\in K[t]_{(t-\theta)}$.
Indeed, we can write $b_j$ more simply, since
\begin{eqnarray*}
\prod_{i=1}^{l-1} \left(1-\frac{t}{\theta^{q^i}}\right) &=& \prod_{i=1}^{l-1} \frac{\theta^{q^i}-t}{\theta^{q^i}} \\
&=& \frac{\prod_{i=1}^{l-1} (\theta^{q^i}-t)}{\theta^{q+\ldots +q^{l-1}}} \\
&=& \frac{\prod_{i=1}^{l-1} (\theta^{q^i}-t)}{\theta^{\frac{q^l-q}{q-1}}}.
\end{eqnarray*}
Hence, after canceling out the $(-1)^{\frac{q^l-q}{q-1}}$, we obtain
\[
 b_j = \prod_{i=1}^{l-1} (\theta^{q^i}-t) \cdot \hdth{j}{\prod_{i=1}^{l-1} (\theta^{q^i}-t)^{-1} }.
\qedhere \]
\end{proof}

The connection of the hyperderivatives $\hdt{n}{\Omega^{-1}}$ with the hyperderivatives $\hdth{n}{\pitilde}$ is given by the
following proposition.

\begin{prop}\label{prop:same-span}
For all $n\geq 0$, the sum $\hdth{n}{\pitilde}+\hdt{n}{\Omega^{-1}}|_{t=\theta}$ lies in the $K$-span of the
elements $\hdt{j}{\Omega^{-1}}|_{t=\theta}$ for $0\leq j<n$. 
\end{prop}

\begin{proof} Combining Lemma \ref{lem:composition-rules} and Lemma \ref{lem:theta-hyperderivatives-of-Omega}, and the identity $\Omega^{-1}|_{t=\theta}=-\pitilde$, one computes:
\begin{eqnarray*}
-\hdth{n}{\pitilde} &=& \sum_{i+j=n} \hdt{i}{\hdth{j}{\Omega^{-1}}}|_{t=\theta} \\
&=&  \sum_{i+j=n}\hdt{i}{b_j\cdot \Omega^{-1}}|_{t=\theta} \\
&=& \left(\sum_{i+j=n}\sum_{i_1+i_2=i} \hdt{i_1}{b_j}\hdt{i_2}{ \Omega^{-1}}\right)|_{t=\theta} \\
&=& \hdt{n}{\Omega^{-1}}|_{t=\theta} + \left(\sum_{i_2<n} \bigl(\sum_{i_1+j=n-i_2} \hdt{i_1}{b_j}\bigr)\hdt{i_2}{ \Omega^{-1}}\right)|_{t=\theta}.
\end{eqnarray*}
\end{proof}

\begin{rem}\label{rem:first-explicit-expression}
We can rewrite the last line using $\Dt$ and $\Dth$. For that we write
\[ B(X):=\Dt(\Omega^{-1})\cdot \Omega = \sum_{j=0}^\infty b_jX^j \in K[t]_{(t-\theta)}\ps{X},\] where $b_j\in K[t]_{(t-\theta)}$ as in Lemma \ref{lem:theta-hyperderivatives-of-Omega}. Then:
\begin{eqnarray*} 
-\Dth(\pitilde) &=& \Dt(\Dth(\Omega^{-1}))|_{t=\theta} = \Dt(B(X)\cdot \Omega^{-1})|_{t=\theta} \\
&=& \Dt(B(X)) \cdot \Dt(\Omega^{-1})|_{t=\theta}.
\end{eqnarray*}
\end{rem}

\begin{proof}[Proof of Thm.~\ref{thm:pitilde-hypertrancendental}]
For the proof, we show that for any $n\geq 0$, the transcendence degree of 
$\bar{K}(\pitilde, \hdth{1}{\pitilde},\ldots, \hdth{n}{\pitilde})$ over $\bar{K}$ is $n+1$.

By \cite[Prop.~6.2]{am:ptmaip}, the matrix
\[   \Psi = \rho_{t,[n]}(\Omega) =  \begin{pmatrix}
\Omega & \hdt{1}{\Omega}  & \hdt{2}{\Omega} & \cdots&\hdt{n}{\Omega} \\
0 & \Omega & \hdt{1}{\Omega} &  \ddots   & \vdots \\ 
\vdots &\ddots  & \ddots & \ddots &   \hdt{2}{\Omega}\\
\vdots & & \ddots  & \Omega &  \hdt{1}{\Omega} \\
0 &  \cdots & \cdots & 0 &  \Omega
\end{pmatrix}  \]
is the rigid analytic trivialization of the dual $t$-motive corresponding to the $n$-th prolongation of the Carlitz module.\\
The transcendence degree of $\bar{K}(t)(\Psi)$ over $\bar{K}(t)$ is the same as the one when adjoining to $\bar{K}(t)$ the entries
of \footnote{The letter $\omega$ denotes the Anderson-Thakur function which is related to $\Omega$ by $\omega=\frac{1}{(t-\theta)\Omega}$.}
\[\rho_{t,[n]}(\omega)=\rho_{t,[n]}\left(\frac{1}{(t-\theta)\Omega}\right)=\rho_{t,[n]}(t-\theta)^{-1}\cdot \Psi^{-1}.\]
This transcendence degree equals $n+1$ by \cite[Thm.~7.2]{am:ptmaip}.

By a theorem of Papanikolas (cf.~\cite[Thm.~5.2.2]{mp:tdadmaicl} and its proof; see also a refinement of Chang in \cite[Thm.~1.2(2)]{cc:nrvabpc}), the transcendence degree of $\bar{K}(t)(\Psi)$ over $\bar{K}(t)$ is the same as the transcendence degree of $\bar{K}(\Psi(\theta))$ over $\bar{K}$.
The latter is the same when adjoining the entries of 
\[ \Psi(\theta)^{-1}= \rho_{t,[n]}(\Omega^{-1})|_{t=\theta}, \]
i.e. adjoining the elements $\Omega^{-1}|_{t=\theta}, \hdt{1}{\Omega^{-1}}|_{t=\theta}, \ldots, \hdt{n}{\Omega^{-1}}|_{t=\theta}$.
Finally by Prop.~\ref{prop:same-span}, the elements $\pitilde, \hdth{1}{\pitilde},\ldots, \hdth{n}{\pitilde}$ span the same 
$\bar{K}$-vector space, hence the fields generated by these sets of elements are the same, and we conclude that the transcendence degree of 
$\bar{K}(\pitilde, \hdth{1}{\pitilde},\ldots, \hdth{n}{\pitilde})$ over $\bar{K}$ is indeed $n+1$.
\end{proof}

\section{Period coordinates of Carlitz tensor powers}\label{sec:tensor-power}

Let $n$ be a positive integer, and let $C^{\otimes n}$ denote the $n$-th Carlitz tensor power.
In \cite{am:ptmaip}, we showed an algebraic independence result for the coordinates of a fundamental period $\pitilde_n=\transp{(z_1,\ldots, z_n)}$ of $C^{\otimes n}$.
It was obtained in a similar manner as for the hyperderivatives of $\pitilde$ in the previous section. Besides using Papanikolas' theorem \cite[Thm.~5.2.2]{mp:tdadmaicl}, the main point was the following identity, where we chose the fundamental period such that its last coordinate $z_n$ equals $\pitilde^n$:

\begin{equation}\label{eq:omega-tensor-power}
\rho_{t,[n-1]}(\Omega^{-n})|_{t=\theta} 
= (-1)^n \cdot\begin{pmatrix}
  z_n &   z_{n-1}  &   z_{n-2} & \cdots &  z_1 \\
0 &   z_n &    z_{n-1} &  \ddots   & \vdots \\ 
\vdots &\ddots  & \ddots & \ddots &     z_{n-2}\\
\vdots & & \ddots  &    z_n &    z_{n-1} \\
0 &  \cdots & \cdots & 0 &    z_n
\end{pmatrix}.
\end{equation}

As one can already imagine, this provides a link between these coordinates and the hyperderivatives of $\pitilde$.

\begin{thm}\label{thm:coordinates-in-k-space}
The coordinates $z_1,\ldots, z_n$ belong to the $K$-vector space generated by the set of ``monomials''
$\left\{ \prod_{j=1}^n \hdth{m_j}{\pitilde} \,\middle|\, \forall j: 0\leq m_j\leq n-1 \right\}. $
\end{thm}

\begin{proof}
Equation \eqref{eq:omega-tensor-power} can be rewritten as:
\[ z_n+z_{n-1}X+\ldots + z_1X^{n-1} \equiv (-1)^n \Dt(\Omega^{-n})|_{t=\theta}  \mod{X^n}. \]
By Remark \ref{rem:first-explicit-expression}, we further have:
\begin{eqnarray*}
(-1)^n \Dt(\Omega^{-n})|_{t=\theta} &=& \Dt(-\Omega^{-1})^n|_{t=\theta} \\
&=& \Dt(B(X))^{-n}\Dth(\pitilde)^n|_{t=\theta} \\
&=& \Dt(B(X))^{-n}|_{t=\theta}\cdot \Dth(\pitilde)^n
\end{eqnarray*}
As the coefficients of $B(X)$ are in $K[t]_{(t-\theta)}$, hence $\Dt(B(X))^{-n}|_{t=\theta}\in K\ps{X}$, and as the non-zero entries of $\Dth(\pitilde)$ up to the coefficients of $X^{n-1}$ are
the elements $\pitilde, \hdth{1}{\pitilde},\ldots,\hdth{n-1}{\pitilde}$, this completes the proof.
\end{proof}

\section{Expressions for the period coordinates}\label{sec:nicer-expressions}

We are going to derive nicer expressions for the coefficients $\Dt(B(X)^{-n})|_{t=\theta}$ occurring in the proof of Theorem \ref{thm:coordinates-in-k-space}. 
We start by deriving a nicer expression for $\Dt(B(X))|_{t=\theta} \mod{X^n}$.

\begin{defn}\label{def:eta}
For $l\geq 0$, we let 
\[ L_l=\prod_{m=1}^l (\theta^{q^m}-\theta)\in K\quad \text{and}\quad
\cL_{l}(t)=\prod_{m=1}^{l} (\theta^{q^m}-t)\in K[t].\]
Furthermore, we define the elements
\[ \eta_l(t)=\prod_{m=1}^l \frac{t^{q^m}-\theta}{\theta^{q^m}-\theta} = \prod_{m=1}^l \left( 1+\frac{(t-\theta)^{q^m}}{\theta^{q^m}-\theta}\right) \in K[t]_{(t-\theta)}, \]
as well as their limit in $K\cs{t-\theta}$,
\[ \eta(t):= \lim_{l\to \infty} \eta_l(t) 
= \prod_{m=1}^\infty \left( 1+\frac{(t-\theta)^{q^m}}{\theta^{q^m}-\theta}\right)\in K\cs{t-\theta}.
\]
\end{defn}

\begin{rem}
The element $\eta(t)$ is the inverse of $\xi_C(t)$ in \cite{mp:latpcmsvgl} where the formula above was a consequence of a different definition (cf.~\cite[Prop.~7.2.2]{mp:latpcmsvgl}). Papanikolas' definition in terms of $\eta(t)$ merely is
\[ \eta(t)=\frac{\Omega(t,\theta)}{\Omega(\theta,\theta)}=-\pitilde\cdot \Omega(t,\theta) \]
viewed in some appropriate ring. This definition
shows its close connection to the Carlitz module. Note that the roles of $t$ and $\theta$ in the numerator are interchanged in that definition.

Using Equation \eqref{eq:omega-tensor-power} and a swap of $t$ and $\theta$, Papanikolas deduced from this definition the equation
\[ z_{n-j} = \hdth{j}{\eta^{-n}\cdot \pitilde^n}|_{t=\theta}. \]
This is the limit version of our identity in Corollary \ref{cor:nice-expression} below.
\end{rem}

\begin{lem}\label{lem:dth-eta-l}
For all $l\in \NN$:
\[\Dth(\eta_l)|_{t=\theta}= \frac{\Dt(\cL_{l}(t))|_{t=\theta}}{\Dth\left(L_{l}\right)} . \]
\end{lem}

\begin{proof}
On one hand, we have
\begin{eqnarray*}
 \Dth(\eta_l)|_{t=\theta}&=&  \Dth\left(\prod_{m=1}^l \frac{t^{q^m}-\theta}{\theta^{q^m}-\theta}\right)\bigg|_{t=\theta}
= \prod_{m=1}^l \frac{t^{q^m}-(\theta+X)}{(\theta+X)^{q^m}-(\theta+X)}\bigg|_{t=\theta} \\
&=& \prod_{m=1}^l \frac{\theta^{q^m}-\theta-X}{\theta^{q^m}+X^{q^m}-\theta-X} .
\end{eqnarray*}
On the other hand,
\begin{eqnarray*}
\frac{\Dt(\cL_{l}(t))|_{t=\theta}}{\Dth\left(L_{l}\right)} &=& \frac{\prod_{i=1}^{l} \Dt(\theta^{q^i}-t)}{\prod_{i=1}^{l} \Dth(\theta^{q^i}-\theta)}\bigg|_{t=\theta} 
 = \frac{\prod_{i=1}^{l} \left(\theta^{q^i}-(t+X)\right)}{\prod_{i=1}^{l} \left((\theta+X)^{q^i}-\theta-X\right)}\bigg|_{t=\theta} \\
 &=&  \prod_{m=1}^l \frac{\theta^{q^m}-\theta-X}{\theta^{q^m}+X^{q^m}-\theta-X}.
\end{eqnarray*}
\end{proof}

\begin{prop}
If $q^l\geq n$, then
\begin{eqnarray*} \Dt(B(X))|_{t=\theta} &\equiv & 
\Dth(\eta_{l-1})|_{t=\theta}  \mod{X^n}.  
\end{eqnarray*}
\end{prop}

\begin{proof}

Let $l\in \NN$ such that $q^l\geq n$. Then from Equation \eqref{eq:b_j-simpler}, we see that
\[ B(X) \equiv \cL_{l-1}(t)\cdot \Dth(\cL_{l-1}(t)^{-1})= \cL_{l-1}(t)\cdot \Dth(\cL_{l-1}(t))^{-1} \mod{X^n},\]
and hence modulo $X^n$ we obtain
\begin{eqnarray*}
 \Dt(B(X))|_{t=\theta} &\equiv & \Dt(\cL_{l-1}(t))|_{t=\theta} \cdot \Dt(\Dth(\cL_{l-1}(t)))^{-1}|_{t=\theta} \\
 &=&  \frac{\Dt(\cL_{l-1}(t))|_{t=\theta}}{\Dth\left(L_{l-1}\right)} =\Dth(\eta_{l-1})|_{t=\theta}, 
\end{eqnarray*}
where we used  Lemma \ref{lem:composition-rules} in the second last step.
\end{proof}

Applying this congruence to the formula obtained in the proof of Thm.~\ref{thm:coordinates-in-k-space}, we get the following corollary.

\begin{cor}\label{cor:nice-expression}
Let $z_1,\ldots, z_n$ be the coordinates of the period of $C^{\otimes n}$ as in Equation \eqref{eq:omega-tensor-power}, and $l\in \NN$ such that $q^l\geq n$. Then
\[ \begin{pmatrix}
  z_n &   z_{n-1}  &    \cdots &  z_1 \\
0 &   z_n &     \ddots   & \vdots \\ 
\vdots & \ddots  &   \ddots  &    z_{n-1} \\
0 &  \cdots & 0 &    z_n
\end{pmatrix} =\rho_{\theta,[n-1]}\left(\eta_{l-1}^{-1}\pitilde\right)^n|_{t=\theta}
= \rho_{\theta,[n-1]}\left(\eta_{l-1}^{-1}\right)^n|_{t=\theta}\cdot \rho_{\theta,[n-1]}\left(\pitilde\right)^n,
\]
or equivalently for $0\leq j\leq n-1$: 
\[  z_{n-j}= \hdth{j}{\eta_{l-1}^{-n}\pitilde^n}|_{t=\theta}. \]
\end{cor}

Papanikolas discovered another nice expression using Anderson-Thakur polynomials and Carlitz factorials (see Thm.~\ref{thm:papanikolas-identity} or the second equation in Corollary \ref{cor:even-nicer-expression}).
He derived it from the equation that we state as first equation in Corollary \ref{cor:even-nicer-expression}. His proof for that equation, however, is quite long and we managed to shorten it immensely. This is due to the first identity in Thm.~\ref{thm:properties-of-eta} (already stated in the introduction as Thm.~\ref{thm:eta-in-intro}) from which the rest is deduced almost instantly.
Nevertheless, the proof of Thm.~\ref{thm:properties-of-eta} builds on ideas from Papanikolas' proof.

Let's recall the necessary definitions. We will use the notation as in Papanikolas' manuscript which differ from the notation in \cite{ga-dt:tpcmzv}.

\begin{defn}\label{def:gamma-et-al}
For $m\geq 0$, let \[ \gamma_m(t)=\prod_{k=1}^m (\theta^{q^m}-t^{q^k})\in \Fq[\theta,t]\] where the empty product (case $m=0$) is defined to be $1$, as usual. Further, we let
\[ D_m = \prod_{k=0}^{m-1} \left(\theta^{q^m}-\theta^{q^{k}}\right)\in \Fq[\theta], \]
and $\Gamma_m\in \Fq[\theta]$ be the Carlitz factorial, defined by 
\[ \Gamma_m = \prod_{j=0}^r D_{j}^{m_j} \]
where $m=m_0+m_1q+\ldots + m_rq^r$ in base-$q$ expansion (i.e. with $m_j\in \{0,\ldots, q-1\}$).

The Anderson-Thakur polynomials $\alpha_n(t)\in K[t]$ are then defined by the generating series
\[ \sum_{n=1}^\infty \frac{\alpha_n(t)}{\Gamma_n}x^{n-1} = \left( 1- \sum_{j=0}^\infty
\frac{\gamma_j(t)}{D_j}x^{q^j}\right)^{-1}\in K[t]\ps{x}. \]
\end{defn}

\begin{rem}
From the definition of the Anderson-Thakur polynomials, one obtains a recursive formula for computing these polynomials: $\alpha_1(t)=1$, and for $n\geq 2$,
\begin{equation}\label{eq:recursion-alpha}
 \frac{\alpha_n(t)}{\Gamma_n}=\sum_{j=0}^{\ell_{n-1}} \frac{\gamma_j(t)}{D_j}\cdot  \frac{\alpha_{n-q^j}(t)}{\Gamma_{n-q^j}}, 
\end{equation}
where $\ell_{n-1}:=\lfloor \log_q(n-1)\rfloor$.\\
Furthermore, one has for $n\geq 1$,
\begin{equation}\label{eq:alpha-q-times} 
\frac{\alpha_{nq}(t)}{\Gamma_{nq}}=\left(\frac{\alpha_n(t)}{\Gamma_n}\right)^q 
\end{equation}
(see \cite[Cor.~7.1.10]{mp:latpcmsvgl}).
\end{rem}

\begin{thm}\label{thm:properties-of-eta}
Let $\eta = \prod_{m=1}^\infty \left( 1+\frac{(t-\theta)^{q^m}}{\theta^{q^m}-\theta}\right)\in K\cs{t-\theta}$ as in Definition \ref{def:eta}.
We have the following identities in $K\cs{t-\theta}$:
\begin{enumerate}
\item \label{item:identity-for-eta}
\begin{equation*} 
 \sum_{j=0}^\infty \frac{\gamma_j(t)}{D_j}\cdot \eta^{q^j} = 1\in K\cs{t-\theta}.
\end{equation*}
\item For all $n\geq 1$, 
\[\eta^{-n} \equiv \frac{\alpha_n(t)}{\Gamma_n} \mod{(t-\theta)^{n+1}}. \]
\item For all $n\geq 1$, and $l\in \NN$ such that $q^l> n$,
\[\eta_l^{-n} \equiv \frac{\alpha_n(t)}{\Gamma_n} \mod{(t-\theta)^{n+1}}. \]
\end{enumerate}
 
\end{thm}

\begin{proof}
The series in \eqref{item:identity-for-eta} is well defined in $K\cs{t-\theta}$, since $\gamma_m(t)$ is divisible by $\theta^{q^m}-t^{q^m}=(\theta-t)^{q^m}$. First, we observe that for $m\geq 1$,
\begin{eqnarray*}
 \frac{\gamma_m(t)}{D_m} &=& \frac{\prod_{k=1}^m (\theta^{q^m}-t^{q^k})}{\prod_{k=0}^{m-1} \left(\theta^{q^m}-\theta^{q^{k}}\right)}
 \quad =\quad \frac{(\theta-t)^{q^m}}{\theta^{q^m}-\theta} \cdot \prod_{k=1}^{m-1} \frac{\theta^{q^m}-t^{q^k}}{\theta^{q^m}-\theta^{q^{k}}}\\
 &=& \frac{(\theta-t)^{q^m}}{\theta^{q^m}-\theta} \cdot \prod_{k=1}^{m-1} \left( 1 + \frac{\theta^{q^k}-t^{q^k}}{\theta^{q^m}-\theta^{q^{k}}} \right) \\
& =& - \frac{(t-\theta)^{q^m}}{\theta^{q^m}-\theta} \cdot \prod_{k=1}^{m-1} \left( 1 + \frac{(t-\theta)^{q^k}}{\theta^{q^k}-\theta^{q^{m}}} \right).
\end{eqnarray*}
So the left hand side of the equation in \eqref{item:identity-for-eta} is
\begin{eqnarray*}
&&\eta \quad
- \quad\sum_{j=1}^\infty \frac{(t-\theta)^{q^j}}{\theta^{q^j}-\theta} \cdot \prod_{k=1}^{j-1} \left( 1 + \frac{(t-\theta)^{q^k}}{\theta^{q^k}-\theta^{q^{j}}} \right) \cdot \prod_{m=1}^\infty \left( 1+\frac{(t-\theta)^{q^{m+j}}}{\theta^{q^{m+j}}-\theta^{q^j}}\right) \\
&=&\eta  \quad
-\quad \sum_{j=1}^\infty \frac{(t-\theta)^{q^j}}{\theta^{q^j}-\theta} \cdot \prod_{k=1}^{j-1} \left( 1 + \frac{(t-\theta)^{q^k}}{\theta^{q^k}-\theta^{q^{j}}} \right) \cdot \prod_{k=j+1}^\infty \left( 1+\frac{(t-\theta)^{q^{k}}}{\theta^{q^{k}}-\theta^{q^j}}\right) \\
&=& \prod_{m=1}^\infty \left( 1+\frac{(t-\theta)^{q^m}}{\theta^{q^m}-\theta}\right) 
- \sum_{j=1}^\infty \frac{(t-\theta)^{q^j}}{\theta^{q^j}-\theta} \cdot \prod_{\substack{k=1\\k\ne j}}^{\infty} \left( 1 + \frac{(t-\theta)^{q^k}}{\theta^{q^k}-\theta^{q^{j}}} \right).
\end{eqnarray*} 
Hence, the only powers of $(t-\theta)$ that occur in this expression are $(t-\theta)^0$ (with coefficient $1$) and
 $(t-\theta)^n$ where $n$ is of the form  $n=q^{m_1}+q^{m_2}+\ldots + q^{m_s}$ with $s\geq 1$ and $0< m_1< m_2<\ldots <m_s$.
The coefficient of such an $n$ is:
\[ \prod_{i=1}^s \frac{1}{(\theta^{q_{m_i}}-\theta)} - \sum_{i=1}^s\frac{1}{(\theta^{q^{m_i}}-\theta)} \prod_{\substack{k=1\\k\ne i}}^s \frac{1}{ (\theta^{q^{m_k}}-\theta^{q^{m_i}})}
\]
By Lemma \ref{lem:lagrange-interpolation} below, this coefficient is zero.

\medskip

The second part is shown by induction on $n$.\\
If $n=1$, we have $\frac{\alpha_1(t)}{\Gamma_1}=1$ and \[\eta^{-1}\equiv 1 \mod{(t-\theta)^q}.\] Since $q\geq 2$, this shows the base case.

If $n>1$, and $n$ is divisible by $q$, i.e. $n=q\cdot m$ with $m\geq 1$, then using the identity \eqref{eq:alpha-q-times} and the induction hypothesis for $m$, 
\[ \eta^{-n}-\frac{\alpha_n(t)}{\Gamma_n} =(\eta^{-m})^q-\left(\frac{\alpha_m(t)}{\Gamma_m}\right)^q 
= \left(\eta^{-m}-\frac{\alpha_m(t)}{\Gamma_m}\right)^q  \equiv 0 \mod {(t-\theta)^{(m+1)\cdot q}} \]
As $(m+1)\cdot q=qm+q>n+1$, we are done in this case.

If $n>1$ is not divisible by $q$, we note that $\ell_{n-1}=\lfloor \log_q(n-1)\rfloor=\lfloor \log_q(n)\rfloor=\ell_n\geq 0$.
By using the induction hypothesis for all $m<n$, as well as the identities \eqref{item:identity-for-eta} and \eqref{eq:recursion-alpha}, we obtain:
\begin{eqnarray*}
\eta^{-n} &=& \sum_{j=0}^\infty \frac{\gamma_j(t)}{D_j} \eta^{-(n-q^j)} 
=\sum_{j=0}^{\ell_{n-1}} \frac{\gamma_j(t)}{D_j} \eta^{-(n-q^j)}+ \sum_{j=\ell_{n-1}+1}^{\infty} \frac{\gamma_j(t)}{D_j} \eta^{-(n-q^j)}\\
&\equiv& \sum_{j=0}^{\ell_{n-1}} \frac{\gamma_j(t)}{D_j} \frac{\alpha_{n-q^j}(t)}{\Gamma_{n-q^j}} = \frac{\alpha_{n}(t)}{\Gamma_{n}}
\mod{(t-\theta)^m}
\end{eqnarray*}
where $m=\min\{ q^{\ell_{n-1}+1}, q^j+(n-q^j)+1 \mid j=0,\ldots, \ell_n \}=\min\{  q^{\ell_n+1}, n+1\} =n+1$.

\medskip

The third part is an immediate consequence of the second part, as $\eta\equiv \eta_l$ modulo $(t-\theta)^{q^l}$, and $q^l\geq n+1$.
\end{proof}

\begin{lem}\label{lem:lagrange-interpolation}
Let $F$ be a field, and let $a_1,\ldots, a_s,b\in F$ be pairwise distinct elements. Then
\[  \prod_{i=1}^s \frac{1}{(a_i-b)} - \sum_{i=1}^s\frac{1}{(a_i-b)} \prod_{\substack{k=1\\k\ne i}}^s \frac{1}{(a_k-a_i)}=0.
\]
\end{lem}

\begin{proof}
By Lagrange interpolation, we have the identity
\[ \sum_{i=1}^s  \prod_{\substack{k=1\\k\ne i}}^s \frac{(a_k-x)}{(a_k-a_i)} =1 \]
in the polynomial ring $F[x]$, since $a_1,\ldots, a_s\in F$ are pairwise distinct.\\
Dividing by $ \prod_{i=1}^s \frac{1}{(a_i-x)}$ and evaluating at $b\not\in \{a_1,\ldots, a_s\}$ leads to the desired result.
\end{proof}

From the previous theorem and Corollary \ref{cor:nice-expression}, we immediately get the following corollary

\begin{cor}\label{cor:even-nicer-expression}
For all $n\geq 1$ and $0\leq j\leq n$, we have
\[ \hdth{j}{\eta^{-n}}|_{t=\theta} = \hdth{j}{\frac{\alpha_{n}(t)}{\Gamma_{n}}}\Big|_{t=\theta}, \]
as well as
\[ z_{n-j}= \hdth{j}{\frac{\alpha_{n}(t)}{\Gamma_{n}}\pitilde^n}\Big|_{t=\theta}, \]
where $z_{n-j}$ is the coordinate of the period of Carlitz tensor power as given in Equation \eqref{eq:omega-tensor-power}.
\end{cor}

\begin{proof}
For all $m>j\geq 0$, and $f\in K\cs{t-\theta}$, the hyperderivative $\hdth{j}{(t-\theta)^m\cdot f}$ is divisible by $(t-\theta)$, and so $\hdth{j}{(t-\theta)^m\cdot f}|_{t=\theta}=0$.
Therefore, if two elements $h,g\in K\cs{t-\theta}$ are congruent modulo $(t-\theta)^{n+1}$, then for all $0\leq j\leq n$:
\[  \hdth{j}{h}|_{t=\theta} =  \hdth{j}{g}|_{t=\theta}. \]
Therefore, by Theorem \ref{thm:properties-of-eta}, we have for all $0\leq j\leq n$:
\[  \hdth{j}{\eta_{l-1}^{-n}}|_{t=\theta} =\hdth{j}{\eta^{-n}}|_{t=\theta} = \hdth{j}{\frac{\alpha_{n}(t)}{\Gamma_{n}}}\Big|_{t=\theta},\]
as well as  (using also Corollary \ref{cor:nice-expression})
\[  z_{n-j}=\hdth{j}{\eta_{l-1}^{-n}\pitilde^n}|_{t=\theta} =\hdth{j}{\eta^{-n}\pitilde^n}|_{t=\theta} = \hdth{j}{\frac{\alpha_{n}(t)}{\Gamma_{n}}\pitilde^n}\Big|_{t=\theta}\]
where as before $q^l>n$.
\end{proof}

\bibliographystyle{alpha}


\begin{thebibliography}{{Mau}18}

\bibitem[AT90]{ga-dt:tpcmzv}
Greg~W. Anderson and Dinesh~S. Thakur.
\newblock Tensor powers of the {C}arlitz module and zeta values.
\newblock {\em Ann. of Math. (2)}, 132(1):159--191, 1990.

\bibitem[BD00]{db-ld:lidddm-II}
W.~Dale Brownawell and Laurent Denis.
\newblock Linear independence and divided derivatives of a {D}rinfeld module.
  {II}.
\newblock {\em Proc. Amer. Math. Soc.}, 128(6):1581--1593, 2000.

\bibitem[Bro99]{db:lidddm-I}
W.~Dale Brownawell.
\newblock Linear independence and divided derivatives of a {D}rinfeld module.
  {I}.
\newblock In {\em Number theory in progress, {V}ol. 1
  ({Z}akopane-{K}o\'scielisko, 1997)}, pages 47--61. de Gruyter, Berlin, 1999.

\bibitem[Cha09]{cc:nrvabpc}
Chieh-Yu Chang.
\newblock A note on a refined version of {A}nderson-{B}rownawell-{P}apanikolas
  criterion.
\newblock {\em J. Number Theory}, 129(3):729--738, 2009.

\bibitem[Con00]{kc:dp}
Keith Conrad.
\newblock The digit principle.
\newblock {\em J. Number Theory}, 84(2):230--257, 2000.

\bibitem[Den95]{ld:dmdt}
Laurent Denis.
\newblock D\'{e}riv\'{e}es d'un module de {D}rinfel\cprime d et transcendance.
\newblock {\em Duke Math. J.}, 80(1):1--13, 1995.

\bibitem[Den97]{ld:iac2}
Laurent Denis.
\newblock Ind\'{e}pendance alg\'{e}brique en caract\'{e}ristique deux.
\newblock {\em J. Number Theory}, 66(1):183--200, 1997.

\bibitem[Den00]{ld:iadpmc}
Laurent Denis.
\newblock Ind\'{e}pendance alg\'{e}brique des d\'{e}riv\'{e}es d'une
  p\'{e}riode du module de {C}arlitz.
\newblock {\em J. Austral. Math. Soc. Ser. A}, 69(1):8--18, 2000.

\bibitem[Mat89]{hm:crt}
Hideyuki Matsumura.
\newblock {\em Commutative ring theory}, volume~8 of {\em Cambridge Studies in
  Advanced Mathematics}.
\newblock Cambridge University Press, Cambridge, second edition, 1989.
\newblock Translated from the Japanese by M. Reid.

\bibitem[{Mau}18]{am:ptmaip}
Andreas {Maurischat}.
\newblock Prolongations of t-motives and algebraic independence of periods.
\newblock {\em {Doc. Math.}}, 23:815--838, 2018.

\bibitem[MP19]{am-rp:iddbcppte}
Andreas {Maurischat} and Rudolph {Perkins}.
\newblock {An integral digit derivative basis for Carlitz prime power torsion
  extensions}.
\newblock In {\em Actes de la conf\'erence ``Analogies between number field and
  function field : algebraic and analytic approaches''}, pages 131--149.
  Besan\c{c}on: Presses Universitaires de Franche-Comt\'e, 2019.

\bibitem[Nam21]{cn:arhpldm}
Changningphaabi Namoijam.
\newblock Algebraic relations among hyperderivatives of periods and logarithms
  of {D}rinfeld modules.
\newblock Preprint available from arXiv at http://arxiv.org/abs/2103.09485,
  2021.

\bibitem[Pap08]{mp:tdadmaicl}
Matthew~A. Papanikolas.
\newblock Tannakian duality for {A}nderson-{D}rinfeld motives and algebraic
  independence of {C}arlitz logarithms.
\newblock {\em Invent. Math.}, 171(1):123--174, 2008.

\bibitem[Pap15]{mp:latpcmsvgl}
Matthew~A. Papanikolas.
\newblock Log-algebraicity on tensor powers of the {C}arlitz module and special
  values of {G}oss {L}-functions.
\newblock Preprint, April 2015.

\bibitem[Wad41]{liw:cqtg}
L.~I. Wade.
\newblock Certain quantities transcendental over {$GF(p^n,x)$}.
\newblock {\em Duke Math. J.}, 8:701--720, 1941.

\end{thebibliography}
\def\cprime{$'$}

\vspace*{.5cm}

\parindent0cm

\end{document}